\documentclass{amsart}
\usepackage{cite}
\usepackage[final]{graphicx}
\usepackage[margin=1.5in,letterpaper]{geometry}
\usepackage{booktabs}
\usepackage{url}
\usepackage[T1]{fontenc}
\usepackage{ae,aecompl}
\usepackage[utf8]{inputenc}

\usepackage{caption}
\usepackage{hyphenat}
\usepackage{tikz}
\usetikzlibrary{matrix,arrows}
\tikzset{cdarrow/.style={auto,
    execute at begin node=$\scriptstyle,execute at end node=$}}

\newread\testin

\graphicspath{{draws/}}

\newcommand{\RR}{\mathbb R}
\newcommand{\CC}{\mathbb C}

\newcommand{\QQ}{\mathbb Q}

\newcommand{\EE}{\mathbb E}
\newcommand{\HH}{\mathbb H}
\newcommand{\MM}{\mathbb M}
\newcommand{\FF}{\mathbb F}
\newcommand{\SSS}{\mathbb S}

\newcommand{\kk}{\mathbf{k}}
\newcommand{\BB}{\mathbf{B}}
\newcommand{\MB}{\mathbf{M}}
\newcommand{\SB}{\mathbf{S}}
\newcommand{\DB}{\mathbf{D}}
\newcommand{\VB}{\mathbf{V}}
\newcommand{\WB}{\mathbf{W}}
\newcommand{\PB}{\mathbf{P}}
\newcommand{\gB}{\mathbf{G}}

\newcommand{\bvec}[1]{{\overrightarrow{#1}}}

\DeclareMathOperator{\spanop}{span}

\theoremstyle{plain}
\numberwithin{equation}{section}
\newtheorem{law}{Law}
\newtheorem{theorem}{Theorem}
\newtheorem{proposition}{Proposition}

\newtheorem{lemma}{Lemma}
\newtheorem{corollary}{Corollary}

\theoremstyle{definition}
\newtheorem{definition}{Definition}

\theoremstyle{remark}

\newtheorem{remark}{Remark}

\newcommand{\Edges}{{\mathcal E}}
\newcommand{\Verts}{{\mathcal V}}
\newcommand{\GG}{{\mathcal G}}

\begin{document}
\title{Generic Global Rigidity in Complex and Pseudo-Euclidean Spaces}

\author[Gortler]{Steven J. Gortler}
\author[Thurston]{Dylan P. Thurston}

\begin{abstract}
In this paper we study the property of generic global rigidity 
for frameworks of graphs embedded in  d-dimensional complex space
and in a d-dimensional pseudo-Euclidean space ($\RR^d$ with a metric
of indefinite signature). We show that a graph is generically globally 
rigid in  Euclidean space iff it is generically globally rigid in 
a complex or pseudo-Euclidean space. We also establish that
global rigidity is always a  generic property of a graph in complex space,
and give a sufficient condition for it to be a generic property in 
a pseudo-Euclidean space. 
Extensions to hyperbolic space are 
also discussed.
\end{abstract}

\date{\today}


\maketitle

\section{Introduction}

The property of generic global rigidity of a graph in d-dimensional
Euclidean space has recently been fully 
characterized~\cite{Connelly05:GenericGlobalRigidity,GHT10}.
It is quite natural to study this property in other spaces as well.
For example, recent work of Owen and Jackson~\cite{owjack} has studied
the number of equivalent realizations of frameworks in $\CC^2$.
In this paper we study the property of generic global rigidity 
of graphs embedded 
in $\CC^d$ as well as graphs embedded in a pseudo Euclidean space 
($\RR^d$  equipped with an 
indefinite metric signature). 

We show that a graph $\Gamma$ is generically globally rigid (GGR) in 
d-dimensional Euclidean space iff $\Gamma$ 
is GGR in d-dimensional complex space.
Moreover, for any metric signature, $s$, 
We show that a graph $\Gamma$ is GGR in 
d-dimensional Euclidean space iff $\Gamma$ is GGR in d-dimensional real space
under the signature $s$. Combining this with results 
from~\cite{connelly2010global} also allows us to equate this property with 
generic global rigidity in hyperbolic space.

In the Euclidean and complex cases, global rigidity can be shown
to be a generic property: a given graph is either generically 
globally rigid, or generically globally flexible. In the 
pseudo Euclidean
(and equivalently the hyperbolic) case, though, we do not know this to be 
true. In this paper we do establish that global rigidity in pseudo Euclidean
spaces is a generic property for graphs that contain a large enough GGR
subgraph (such as a d-simplex).

\subsection*{Acknowledgments}
We would  like to thank 
Robert Connelly, 
Bill Jackson,
John Owen, 
Louis Theran, and
for helpful
conversations and suggestions.
We would especially like to thank
Walter Whiteley for sharing with us 
his explanation of the Pogorelov map.
\section{Initial Definitions}
\label{sec:definitions}

\begin{definition}
We equip $\RR^d$ with 
pseudo Euclidean metric in order to measure lengths.
The metric is specified with a non negative integer $s$
that determines  how many of its 
coordinate directions are subtracted from the total.
The squared length of a vector $\vec{w}$
is $|\vec{w}|^2:= -\sum_{i=1}^s \vec{w}_i^2 + \sum_{i=s+1}^d \vec{w}_i^2 $.
We will use the symbol 
$\SSS^d$ to denote the space $\RR^d$ equipped with some
fixed  metric $s$.
If $s=0$, 
we  have the Euclidean metric and the space may be denoted $\EE^d$.

For complex space, 
The squared length of a vector $\vec{w}$
in $\CC^d$ is
$|\vec{w}|^2 := \sum_{i} \vec{w}_i^2$.
Note here that we do not use conjugation, and 
thus vectors have complex squared lengths. 
(The use of conjugation would essentially reduce
d-dimensional complex rigidity questions to 2d-dimensional Euclidean
questions).
\end{definition}

\begin{definition}\label{def:config-space}
  A \emph{graph}~$\Gamma$ is a set of $v$ vertices $\Verts(\Gamma)$
  and $e$ edges~$\Edges(\Gamma)$, where $\Edges(\Gamma)$ is a set of\
  two\hyp element subsets of $\Verts(\Gamma)$.  We will typically drop
  the graph~$\Gamma$ from this notation.

For $\FF \in \{\EE,\SSS,\CC\}$,
  a \emph{configuration} of the vertices $\Verts(\Gamma)$ of a
  graph in $\FF^d$  is 
  a mapping~$p$ from $\Verts(\Gamma)$ to~$\FF^d$.
Let $C_{\FF^d}(\Verts)$ 
be the space of configurations in $\FF^d$.

For $p \in C_{\FF^d}(\Verts)$ with 
$u
  \in \Verts(\Gamma)$, 
we write $p(u) \in \FF^d$ for the image of $u$ under~$p$.

A \emph{framework} 
$\rho = (p,\Gamma)$ of a graph is the pair of a 
graph and a configuration
of its vertices.  $C_{\FF^d}(\Gamma)$ 
is the space of frameworks $(p,
\Gamma)$ with graph $\Gamma$ and configurations in $\FF^d$.

We may also
write $\rho(u)$ for $p(u)$ where $\rho = (p,\Gamma)$ is a framework of the
configuration~$p$.
\end{definition}

\begin{definition}
Two frameworks $\rho$ and $\sigma$ in 
$C_{\FF^d}(\Gamma)$ are \emph{equivalent} if for all $\{t,u\} \in \Edges$
we have $|\rho(t)-\rho(u)|^2=|\sigma(t)-\sigma(u)|^2$.
\end{definition}

\begin{definition}
Two configurations $p$ and $q$ in 
$C_{\FF^d}(\Verts)$ are \emph{congruent} if for all vertex pairs,  $\{t,u\}$,
we have $|p(t)-p(u)|^2=|q(t)-q(u)|^2$.

Two configurations $p$ and $q$ in 
$C_{\FF^d}(\Verts)$ are \emph{strongly congruent} 
if they are related by a translation composed with
an element of the orthogonal group of $\FF^d$. 
\end{definition}

\begin{remark}
In $\EE^d$, there is no difference between congruence and strong congruence.
In other spaces, though, there
can be some subtle differences. For the simplest example, in $\CC^2$, the
vectors $(0,0)$ and 
$(i,1)$ both have zero length, but are not related by a complex
orthogonal transform. 
Such non-zero vectors with zero squared length are called 
\emph{isotropic}.
Thus the framework  made up of a single 
edge connecting
a vertex at the origin to a vertex at $(i,1)$ is congruent to the framework
with both vertices at the origin, but the two frameworks are not
strongly congruent.

Fortunately, these differences are easy to avoid;
for example,  congruence and strong congruence
coincide 
for points with 
a d-dimensional affine span. These notions 
 will also coincide when there are 
fewer than $d+1$ points, as long as the points are in 
affine general position.
For more details,
see Section~\ref{app:cong}.
\end{remark}

We can now, finally, define global rigidity and flexibility.

\begin{definition}
A framework
$\rho  \in C_{\FF^d}(\Gamma)$ is
\emph{globally rigid} in $\FF^d$
if, for any
other framework $\sigma  \in C_{\FF^d}(\Gamma)$ to which 
$\rho$ 
is equivalent, 
we also have that 
$\rho$ is 
congruent to
 $\sigma$.
Otherwise we say
that $\rho$ is \emph{globally flexible} in $\FF^d$.
\end{definition}

\begin{definition}
  \label{def:generic}
A configuration~$p$
in $C_{\FF^d}(\Verts)$
is generic if the 
coordinates do
  not satisfy any non-zero algebraic equation with rational
  coefficients.  We call a framework generic if its configuration is
  generic.
(See  Section~\ref{app:alg} for more background on (semi) algebraic sets and
genericity).
\end{definition}

\begin{definition}
A graph $\Gamma$ is generically globally rigid (resp. flexible)
in $\FF^d$ 
if all generic frameworks in 
$C_{\FF^d}(\Gamma)$ 
are globally rigid (resp. flexible). These properties
are abbreviated GGR and GGF.
\end{definition}

\begin{definition}
A property is \emph{generic} 
if, for every graph, either all generic frameworks in $C_{\FF^d}(\Gamma)$ 
have the property
or none do.  
For instance, global rigidity in
$\EE^d$ is a  generic property of a graph~\cite{GHT10}.  So in this
case, if a graph is not GGR, it must be GGF.
\end{definition}

\section{Complex Generic Global Rigidity}
Our main theorem in this section is 
\begin{theorem}
\label{thm:complex}
A graph $\Gamma$ 
is generically globally rigid in $\CC^d$
iff it
is generically globally rigid in $\EE^d$.
\end{theorem}

\begin{remark}
\label{rem:chevy}
This fully describes the generic situation for complex frameworks as
it is easy to see that 
generic global rigidity in $\CC^d$ is a generic property of a graph.

Recall that a  complex  algebraically \emph{constructible 
set} is a finite Boolean combination of 
complex algebraic sets.
Also,  an irreducible complex algebraic set $V$ cannot have two disjoint
constructible subsets with the same dimension as $V$.

Chevalley's theorem states 
that the image under a polynomial map of a complex
algebraically constructible set, all defined over $\QQ$,
is also 
a complex algebraically constructible set defined over 
$\QQ$~\cite[Theorem 1.22]{basu2006algorithms}.
Chevalley's theorem 
allows one to apply elimination, effectively replacing
all quantifiers in a Boolean-algebraic expression
with algebraic equations and Boolean set operations.

Now, let us assume $\Gamma$ 
is locally 
rigid in $\CC^d$.
We can partition 
$C_{\CC^d}(\Gamma)$ such that in each part, $P_n$ , all of the frameworks
have the same number, $n$, of  equivalent and non-congruent frameworks.
In light of Chevalley's theorem, each of these parts is constructible.
And exactly one of them, $P_{n_0}$, must be  of full dimension. 
This part contains all of the generic points and represents the generic
behavior of the framework. If $n_0=1$ then the graph is GGR, while if
$n_0>1$ then it must be GGF.
\end{remark}

\subsection{=> of Theorem~\ref{thm:complex}}
\label{sec:complex1}
The implication from Complex to Euclidean GGR follows almost
directly from their definitions.
For this argument  we model 
each Euclidean framework $\rho$  in $C_{\EE^d}(\Verts)$ 
as a Complex framework $\rho_\CC$ in 
$C_{\CC^d}(\Verts)$ 
that happens to have all purely real coordinates.
Clearly, for such configurations, the complex squared length
measurement coincides with the Euclidean metric on real configurations.

\begin{proof}
Let $\rho$ be a generic framework in 
$C_{\EE^d}(\Gamma)$ and let $\rho_\CC$ be its corresponding real valued
framework
in 
$C_{\CC^d}(\Gamma)$. By our definitions,
$\rho_\CC$ is also generic when thought of as complex
framework.

Since $\Gamma$ is  generically globally rigid in $\CC^d$,
$\rho_\CC$ can have no equivalent and non-congruent framework in
  $C_{\CC^d}(\Gamma)$,
and thus it has no real valued,
equivalent and non-congruent  framework in 
 $C_{\CC^d}(\Gamma)$.
Thus $\rho$ has no equivalent and non-congruent framework in 
 $C_{\EE^d}(\Gamma)$.
\end{proof}

\subsection{<= of Theorem~\ref{thm:complex}}
\label{sec:complex2}

For the other direction of Theorem~\ref{thm:complex},
we start with a complex version of a theorem by 
Connelly~\cite{Connelly05:GenericGlobalRigidity}:
\begin{theorem}
\label{thm:comcon}
Let $\rho$ be a generic framework in 
$C_{\CC^d}(\Gamma)$. 
If $\rho$ has a complex equilibrium stress matrix
of rank $v-d-1$, then $\Gamma$ is generically globally rigid in $\CC^d$.
\end{theorem}
\begin{proof}
The proof of the complex version
of this theorem follows identically to Connelly's proof of the 
Euclidean version.
In particular, the proof shows that any framework with the same complex 
squared edge
lengths as $\rho$ must be strongly congruent, and thus congruent to it.
\end{proof}
(The interested reader can 
see~\cite{Connelly05:GenericGlobalRigidity} for the definition of an equilibrium stress matrix).

Next, we recall a theorem from Gortler, Healy and Thurston~\cite{GHT10}
\begin{theorem}
\label{thm:ght}
Let $\rho$ be a generic framework in 
$C_{\EE^d}(\Gamma)$ with at least $d+2$ vertices.
If $\rho$ does not have a real equilibrium stress matrix
of rank $v-d-1$, then $\Gamma$
is generically globally flexible in $\EE^d$.
Moreover, there must be an even number of noncongruent frameworks
with the same squared edge lengths as
$\rho$
in $\EE^d$.
\end{theorem}

And now we can prove this direction of our Theorem.
\begin{proof}
From Theorem~\ref{thm:comcon}, if
$\Gamma$ 
is not generically globally rigid in $\CC^d$, 
there is no  generic framework in
$C_{\CC^d}(\Gamma)$ that has
a 
complex equilibrium stress matrix of rank $v-d-1$.
Thus 
there can be 
no real valued and generic framework in 
$C_{\CC^d}(\Gamma)$
 with 
complex equilibrium stress matrix of rank $v-d-1$, and thus  
no generic framework in 
$C_{\EE^d}(\Gamma)$ 
with a complex or
real equilibrium stress matrix of rank $v-d-1$.
Thus from Theorem~\ref{thm:ght}, $\Gamma$ is 
generically globally flexible in
$\EE^d$.
\end{proof}

\section{Pseudo Euclidean Generic Global Rigidity: Results}
\label{sec:sudo}

Our main theorem on pseudo Euclidean generic global rigidity is as 
follows:
\begin{theorem}
\label{thm:sudo}
For any pseudo Euclidean space $\SSS^d$, 
a graph $\Gamma$ 
is generically globally rigid in $\EE^d$
iff 
it is generically globally rigid in $\SSS^d$.
\end{theorem}

Unfortunately we do not know if 
generic global rigidity is a generic property in 
$\SSS^d$. 
It is conceivable that there are some graphs that are not 
GGR in $\SSS^d$
but that do have \emph{some} generic frameworks that are
globally rigid in $\SSS^d$. 
We leave this as an open question. We do have the following
partial result

\begin{theorem}
\label{thm:sudoGP}
If a graph $\Gamma$ 
is not GGR in 
$\SSS^d$ and it has 
a GGR subgraph $\Gamma_0$ with $d+1$ 
or more vertices, then $\Gamma$ 
must be GGF in $\SSS^d$.
\end{theorem}

\section{=> of Theorem~\ref{thm:sudo}}
\label{sec:sudo1}

This argument is essentially identical to that of Section~\ref{sec:complex1}.

\begin{definition}
Given a pseudo Euclidean space $\SSS^d$ with signature $s$, 
we model 
each configuration $\rho \in C_{\SSS^d}(\Verts)$ 
as a Complex configurations 
$\rho_\CC \in C_{\CC^d}(\Verts)$ 
that happens to have the first $s$ of its coordinates purely imaginary
and the remaining $d-s$ of its coordinates purely real. We call this
an \emph{s-signature, real valued complex configuration}.
We will shorten this to simply an 
\emph{s-valued  configuration}.

It is easy to verify that for such configurations, the complex squared length
measurement coincides with the metric on $\SSS^d$. 
\end{definition}

And now we can prove this direction of our Theorem.
\begin{proof}
Let $\rho$ be a generic framework in 
$C_{\SSS^d}(\Gamma)$. We model this with  $\rho_\CC$,
an
s-valued complex framework  in
$C_{\CC^d}(\Gamma)$. 

$\rho_\CC$ must be a generic framework in $C_{\CC^d}(\Gamma)$. 
For suppose there is a non-zero polynomial $\phi_\CC$
with rational coefficients, that vanishes on 
$\rho_\CC$. Then there is a polynomial $\phi$ with coefficients
in $\QQ(i)$ that vanishes on the real coordinates of $\rho$. 
Let $\bar{\phi}$ be the polynomial
 obtained by taking the conjugate of every coefficient in $\phi$,
and let $\psi := \phi * \bar{\phi}$. Then $\psi$ is non zero
and vanishes on 
$\rho$. Since $\psi$  is fixed by conjugation, it
 has coefficients in $\QQ$. This polynomial
would make $\rho$ non generic, leading to a contradiction.

Since $\Gamma$ is  generically globally rigid in $\EE^d$, from 
Theorem~\ref{thm:complex} it is also generically globally rigid in 
$\CC^d$. Thus $\rho_\CC$
can have no equivalent and non-congruent framework in
  $C_{\CC^d}(\Gamma)$,
and thus it can have  no s-valued,
equivalent and non-congruent  framework in 
 $C_{\CC^d}(\Gamma)$. 
Thus $\rho$ can have no equivalent and non-congruent framework in 
 $C_{\SSS^d}(\Gamma)$.
\end{proof}

\section{<= of Theorem~\ref{thm:sudo}}
\label{sec:sudo2}

\begin{remark}
\label{rem:g3}
For this proof, we cannot apply the same reasoning as 
section~\ref{sec:complex2},
as many of the stress matrix arguments and conclusions
from~\cite{GHT10} simply
do not carry over
to pseudo Euclidean spaces.
Indeed, Jackson and Owen~\cite{owjack}
have found a graph, they call $G_3$, 
that is GGF in $\EE^2$, but for which
there is always an \emph{odd} number of equivalent realizations in 
2-dimensional Minkowski space. Moreover, it is not even clear that
for general pseudo Euclidean spaces of dimension 3 or greater, 
the ``number of
equivalent realizations mod $2$'' is even a generic property.
\end{remark}

For this direction, we will
show the contrapositive: namely, if there is 
a generic Euclidean framework that is not globally rigid, then there
must be a generic framework in $\SSS^d$ that is not globally rigid. 
To do this, we will apply 
a basic construction by Saliola and Whiteley~\cite{waltPC} 
that takes a pair of equivalent
Euclidean frameworks and produces a pair of equivalent 
frameworks in the desired space
 $C_{\SSS^d}(\Gamma)$. 
Whiteley refers to this recipe as a generalized
Pogorelov map~\cite{waltPC}. 

\begin{definition}
\label{def:pog}
Let $P$ 
be the map from pairs of frameworks in 
 $C_{\EE^d}(\Gamma)$
 to pairs
of frameworks in 
 $C_{\SSS^d}(\Gamma)$
 defined as follows:

Step 1: Let $\rho$ and $\sigma$ be two frameworks in 
$\EE^d$. Take their average to obtain $a := \frac{\rho+\sigma}{2}$.
Take their difference to obtain $f := \frac{\rho-\sigma}{2}$.

Step 2: Let $\tilde{a}$ be the framework  in 
 $C_{\SSS^d}(\Gamma)$ 
with the same (real) coordinates of $a$.
Let $\tilde{f}$ be defined by negating the first $s$ of the  coordinates in $f$.

Step 3: Finally, set
$P(\rho,\sigma) := (\tilde{\rho},\tilde{\sigma})$ where
$\tilde{\rho} := \tilde{a} + \tilde{f}$ and 
$\tilde{\sigma} :=  \tilde{a} - \tilde{f}$. 
\end{definition}

The Pogorelov map is useful due to the following~\cite{waltPC}:
\begin{theorem}
\label{thm:pog}
Let $\rho$ and $\sigma$ be two equivalent frameworks in 
$C_{\EE^d}(\Gamma)$. 
Then $P(\rho,\sigma)$ are a pair of 
equivalent frameworks in 
$C_{\SSS^d}(\Gamma)$. 
\end{theorem}
\begin{proof} 
Using the notation of Definition~\ref{def:pog} we see the following.

Step 1: 
From the \emph{averaging principal}~\cite{Connelly91:GenericGlobalRigidity},
$a$ must be infinitesimally flexible
with flex $f$.

Step 2: 
$\tilde{f}$ must be an infinitesimal flex for $\tilde{a}$
in $C_{\SSS^d}(\Gamma)$~\cite{saliola2007some}.

Step 3: 
From the 
\emph{flex-antiflex principal}~\cite{Connelly91:GenericGlobalRigidity} 
(also sometimes called the de-averaging principal),
$\tilde{\rho}$ must be equivalent to $\tilde{\sigma}$ in 
 $C_{\SSS^d}(\Gamma)$.
\end{proof}

\begin{remark}
\label{rem:cswap}
It is, perhaps, interesting to note that in our case,
the map has the very simple form of ``coordinate swapping''.
In particular, it is an easy calculation to see that 
$\tilde{\rho}$ will be made up of
the first $s$ coordinates of $\rho$ and the remaining coordinates of $\sigma$,
while
$\tilde{\sigma}$ will be made up of
the first $s$ coordinates of $\sigma$ and the remaining coordinates of $\rho$.
It is also an simple calculation to directly verify, without using 
the averaging principle, 
 that coordinate swapping will
map pairs of equivalent Euclidean frameworks to pairs of equivalent
frameworks in $C_{\SSS^d}(\Gamma)$.
\end{remark}

Additionally, we can ensure 
that
$\tilde{\rho}$ is not  congruent to $\tilde{\sigma}$.

\begin{lemma}
\label{lem:pnoncong}
Let $\rho$ and $\sigma$ be two equivalent frameworks in 
 $C_{\EE^d}(\Gamma)$. 
And let $(\tilde{\rho},\tilde{\sigma}) := P(\rho,\sigma)$.
Then 
$\rho$ and $\sigma$ are  congruent 
in  $C_{\EE^d}(\Gamma)$
iff 
$\tilde{\rho}$ and $\tilde{\sigma}$ are congruent
in  $C_{\SSS^d}(\Gamma)$. 
\end{lemma}
\begin{proof}
Congruence between configurations is the same as equivalence between
complete graphs over these configurations. Thus this property must
map across the Pogorelov map (which does not depend on the edge set),
and its inverse.
\end{proof}

\subsection{Genericity}

The main (annoyingly) difficult 
technical issue left is to show that this construction can 
create a \emph{generic} framework in 
 $C_{\SSS^d}(\Gamma)$ that is globally flexible.
A priori, it is conceivable that the image of the Pogorelov map, acting
on all pairs of equivalent and non-congruent Euclidean frameworks, 
can only produce pseudo Euclidean configurations that lie on some 
subvariety of $C_{\SSS^d}(\Gamma)$. 
In this section, we rule this possibility out.

In this discussion,
we will assume that $\Gamma$ is generically locally rigid
(otherwise we are done), but that it is not GGR in $\EE^d$.

\begin{definition}
Let $E^+$ ('E' for 'equivalent')
be the algebraic subset of 
 $C_{\EE^d}(\Gamma) \times  C_{\EE^d}(\Gamma)$
consisting of 
pairs of equivalent tuples.
Let $C^+$ ('C' for 'congruent')
be the algebraic subset of 
 $C_{\EE^d}(\Gamma) \times  C_{\EE^d}(\Gamma)$
consisting of 
pairs of congruent tuples.
Let $\pi_1$ 
be the projection from a pair of frameworks
onto its first factor.
\end{definition}

\begin{definition}
\label{def:E}
Since $\Gamma$ is not GGR in $\EE^d$,
$\dim(\pi_1(E^+\backslash C^+))= v*d$  and so $E^+$ must have 
at least one  irreducible component 
$E$, with 
$\dim(\pi_1(E))= v*d$ and such that 
it contains
at least 
one tuple of non-congruent frameworks.
We choose one such component and call it $E$.
As per Remark~\ref{rem:ext}, $E$ must be defined
over some algebraic extension of $\QQ$. Thus if $e$ is generic in $E$, 
then $\pi_1(e)$ is a generic framework in $C_{\EE^d}(\Gamma)$.
\end{definition}

\begin{lemma}
\label{lem:gpnoncong}
Let $e := (\rho,\sigma) 
\in E$ be generic. Then $\rho$ is not congruent to $\sigma$.
\end{lemma}
\begin{proof}
Congruence is a relation that can be expressed with polynomials over $\QQ$.
By our assumptions on $E$, these polynomials do not vanish identically 
over $E$. 
\end{proof}

\begin{lemma}
\label{lem:dimE}
  The (real) dimension of $E$ is $v*d+\binom{d+1}{2}$.
  Moreover, if $(\rho,\sigma)$ is generic in $E$, then 
for all $\sigma^c$ in $[\sigma]$, the equivalence class of $\sigma$ under
translations and rotations,
$(\rho, \sigma^c)$ must  be in
$E$~\footnote{This lemma and proof contains a correction from the published
version}.
\end{lemma}
\begin{proof}
We will pick a generic $e = (\rho,\sigma) \in E$, and look at the 
dimension of the  fiber  $\pi_1^{-1}(\rho)$ near this point $e$.
(By considering only this neighborhood, we can avoid dealing with
any non-smooth points of $E$, and thus can view this as a 
smooth map between
manifolds).
The dimension of $E$ must be the sum of the dimension of the span
of $\pi_1(E)$, which is $v*d$, 
and the dimension of 
this fiber.

Since $e$ is generic in $E$, 
$\rho$ must be  generic in  $C_{\EE^d}(\Gamma)$.
Thus, from Lemma~\ref{lem:equivgen} (below),
$\sigma$ 
must be locally rigid and with 
non degenerate affine span.
Thus its equivalence class  has
dimension $\binom{d+1}{2}$.

Since $e$ is generic in $E$, from
Lemma~\ref{lem:onlyone}, 
all nearby points in $E^+$ must, in fact,
lie in $E$. In particular, 
for $\sigma^c \in [\sigma]$ and close to
$\sigma$, 
the point
$(\rho, \sigma^c)$ must  be in
$E$.
Thus the dimension of the  fiber in $E$ near $e$ must be $\binom{d+1}{2}$. 
This gives us the desired dimension.

Moreover, 
since $E$ is algebraic, 
for any $\sigma^c \in [\sigma]$,
the point
$(\rho, \sigma^c)$ must  be in
$E$.
This follows from the facts that,
the orbit of $\sigma$ under translations and rotations
is irreducible, 
and
that the
(Zariski) closure of a subset must be a subset of the closure.

\end{proof}

\begin{corollary}
\label{cor:pi2}
Let $\pi_2$ be the projection of a pair onto its second factor. 
  The (real) dimension of $\pi_2(E)$ is $v*d$. And if $e$ is generic
in $E$, $\pi_2(e)$ is generic in $C_{\EE^d}(\Gamma)$.
\end{corollary}

To study the behavior of $P$ on $E$, we move our discussion over
to complex space.

\begin{definition}
Let $E_\CC^+$ 
be the algebraic subset of 
 $C_{\CC^d}(\Gamma) \times  C_{\CC^d}(\Gamma)$
consisting of 
pairs of equivalent tuples.
Let $E_\CC$ 
be any component of 
$E_\CC^+$ that includes $E$.
(This can be done as the complexification of $E$
must be irreducible - see 
Definition ~\ref{def:cxify}).
From 
Corollary~\ref{cor:CE}, below, we will also soon see that there is only
one such component.
\end{definition}

\begin{lemma}
\label{lem:dimCE}
 The (complex) dimension of $E_\CC$ is $v*d+\binom{d+1}{2}$.
\end{lemma}
\begin{proof}
$E_\CC$ includes the complexification of $E$
(see Definition~\ref{def:cxify}), and so by assumption, 
the complex dimension of $\pi_1(E_\CC)$ must be at least $v*d$,
and thus must be equal to $v*d$.
We can then follow the 
proof of Lemma~\ref{lem:dimE} to establish the complex dimension
of the generic $\pi_1$ fibers of $E_\CC$
\end{proof}

\begin{corollary}
\label{cor:CE}
$E_\CC$ is \emph{the} complexification of $E$.
A generic point of $E$ is generic in $E_\CC$. 
\end{corollary}
\begin{proof}
By assumption, $E_\CC$ is irreducible and contains $E$.
Moreover the complex dimension of $E_\CC$ equals the real dimension 
of $E$. Thus $E_\CC$ cannot be larger than the complexification of $E$.
Genericity carries across complexification
(see Definition~\ref{def:cxify}).
\end{proof}

To study $P$, we 
will look at a complex Pogorelov map  $P_\CC$, that 
essentially reproduces the behavior of $P$ when restricted to real input.
In particular, this
map will take
real valued complex pairs, to s-valued complex pairs.
We define  $P_\CC$
 as the composition of some very simple
maps.

\begin{definition}
Let $H_\CC$, (a Haar like transform) be the
invertible map from 
$(\rho_\CC,\sigma_\CC)$, a pair 
of frameworks in 
$C_{\CC^d}(\Gamma)$,
to the pair
$(\frac{\rho_\CC+\sigma_\CC}{2},
\frac{\rho_\CC-\sigma_\CC}{2})$. 

Let $S_\CC$ be the the invertible map that takes 
$(a_\CC,f_\CC)$,
a pair of frameworks in 
$C_{\CC^d}(\Gamma)$,
to the pair 
$(\tilde{a}_\CC,\tilde{f}_\CC)$,
where the $\tilde{a}_\CC$ is obtained from $a_\CC$ by 
multiplying its first $s$ coordinates by $i$, while
$\tilde{f}_\CC$ is obtained from $f_\CC$ by multiplying its first $s$ 
coordinates by $-i$.

$H_\CC^{-1}(\tilde{a}_\CC,\tilde{f}_\CC)$,
the inverse Haar map,
is simply 
$(\tilde{a}_\CC+\tilde{f}_\CC,\tilde{a}_\CC-\tilde{f}_\CC)$.

Given this, $P_\CC := H_\CC^{-1} \circ S_\CC \circ H_\CC$.

This complex Pogorelov map coincides with the real map described above. 
In particular 
suppose  $\rho$ and $\sigma$ are in  $C_{\EE^d}(\Gamma)$,
and suppose 
$\rho_\CC$ and $\sigma_\CC$ are the 
corresponding real valued frameworks in $C_{\CC^d}(\Gamma)$.
Let  
$(\tilde{\rho},\tilde{\sigma}) := P(\rho,\sigma)$ 
and
$(\tilde{\rho}_\CC,\tilde{\sigma}_\CC) := P_\CC(\rho_\CC,\sigma_\CC)$.
Then 
$\tilde{\rho}_\CC$ and $\tilde{\sigma}_\CC$
are the s-valued complex representations of
$\tilde{\rho}$ and $\tilde{\sigma}$.
\end{definition}

\begin{figure}
  \centering
  \includegraphics[width=4in]{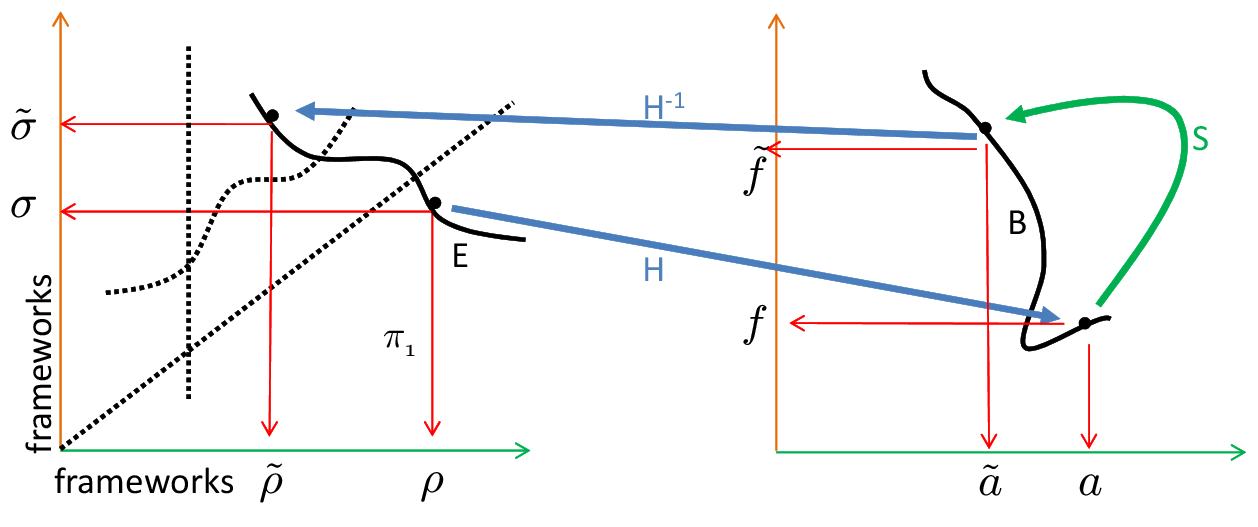}
  \caption{
Left: The space of pairs of complex frameworks. 
(All $\CC$ subscripts are dropped for clarity).
The locus of equivalent pairs, $E_\CC^+$,
is shown in solid and
dotted black. At least one component, $E_\CC$, shown in solid black, 
has the property that 
$\dim(\pi_1(E_\CC))=v*d$.
Right: The space of pairs of complex frameworks.
The variety $B_\CC := H_\CC(E_\CC)$ is made up of some frameworks 
and their flexes. (The image under $H_\CC$ of the other components of 
$E_\CC^+$ is not shown).
The map $S_\CC$ maps $B_\CC$ to itself, and thus
the Pogorelov map is an automorphism of $E_\CC$.
}
  \label{fig:maps}
\end{figure}

Clearly $P_\CC$ 
maps $E_\CC^+$ to itself. But a priori, it might map 
the component $E_\CC$ 
to some other component of $E_\CC^+$,
and this other component might project 
under $\pi_1$ and $\pi_2$
onto a subvariety of (non generic) frameworks
$C_{\CC^d}(\Gamma)$.
Our goal will
be to show that this does not happen; 
instead $E_\CC$ maps to  itself under $P_\CC$.
As this map preservers genericity,
and generic points of $E_\CC$
project under $\pi_1$ to generic frameworks in 
$C_{\CC^d}(\Gamma)$, we will then be done. (See 
Figure~\ref{fig:maps}).

\begin{definition}
Let 
$B_\CC:=(H_\CC(E_\CC))$, ('B' for 'bundles' of flexes over frameworks).
Since $B_\CC$ is isomorphic to $E_\CC$, it too must be an algebraic set.
For any $(a_\CC,f_\CC)\in B_\CC$, from the averaging principle, 
$f_\CC$ is an infinitesimal flex for $a_\CC$.
$B_\CC$ is irreducible
(Lemma~\ref{lem:irrIm}). 
And if $e_\CC$ is generic in $E_\CC$, 
$H_\CC(e_\CC)$
(from Lemma~\ref{lem:image-generic})   must be generic 
in $B_\CC$.
\end{definition}

\begin{lemma}
\label{lem:closeb}
Let 
$b_\CC \in B_\CC$ be generic.
Let $b_\CC' := (a_\CC',f_\CC')$ be a nearby tuple in 
 $C_{\CC^d}(\Gamma) \times  C_{\CC^d}(\Gamma)$
such that $f'_\CC$ is an infinitesimal flex for $a_\CC'$.
Then $b_\CC' \in B_\CC$.
\end{lemma}
\begin{proof}
The tuple, 
$e_\CC := H_\CC^{-1}(b_\CC)$, is generic in $E$.
From the flex/antiflex principal,
$(\rho_\CC',\sigma_\CC') := e_\CC' := H_\CC^{-1}(a_\CC',f_\CC')$ 
must be an equivalent
pair of frameworks and thus in $E_\CC^+$, and 
$e_\CC'$ must be near $e_\CC$.
From Lemma~\ref{lem:onlyone},
all nearby points in $E_\CC^+$ must, in fact,
lie in $E_\CC$. 
Thus $e_\CC'$ must be in $E_\CC$, and from our definitions, 
$H_\CC(e_\CC') = b_\CC'$ must  be in $B_\CC$.
\end{proof}

\begin{definition}
Let $(a_\CC,f_\CC)=b_\CC$ be a pair of framework in $C_{\CC^d}(\Gamma)$.
One can apply \emph{coordinate scaling} to $b_\CC$ by multiplying 
one chosen coordinate (out of the $d$ coordinates in $\CC^d$)
of all the vertices in $a_\CC$ by 
some complex scalar $\lambda$ and the corresponding coordinate
in all the vertices in 
$f_\CC$ by $1/\lambda$.

\end{definition}
\begin{lemma}
\label{lem:inv}
The set $B_\CC$ is invariant to 
coordinate scaling. 
\end{lemma}
\begin{proof}

Let  $(a_\CC,f_\CC) = b_\CC \in B_\CC$ be generic. $f_\CC$ 
is an infinitesimal  flex for $a_\CC$.
Let us  apply coordinate scaling to $b_\CC$ with a 
scalar
$\lambda$ close to $1$ and let us denote the result by 
$b_\CC' = (a_\CC',f_\CC')$.
Looking at the effect of the rigidity matrix, we see that 
$f_\CC'$ must be an infinitesimal flex
for $a_\CC'$, and from Lemma~\ref{lem:closeb} must be in $B_\CC$.

This means that $B_\CC$ is invariant to nearly-unit 
coordinate scaling. 
Since $B_\CC$ is algebraic, it must thus be invariant
to all coordinate scaling.  (This follows from the fact that the
(Zariski) closure of a subset must be a subset of the closure).
\end{proof}

\begin{corollary}
\label{cor:auto}
$S_\CC$ is an automorphism of $B_\CC$.  
Thus $P_\CC$ is an automorphism of $E_\CC$.
Thus if $e_\CC \in E_\CC$ is generic, 
then $P_\CC(e_\CC)$ is generic in $E_\CC$ and both
$\pi_1(P_\CC(e_\CC))$ and
$\pi_2(P_\CC(e_\CC))$ 
are generic
in $C_{\CC^d}(\Gamma)$.
\end{corollary}

With this we can 
finish the proof 
of this direction
of Theorem~\ref{thm:sudo}.

\begin{proof}
Assume that $\Gamma$ is not GGR in $\EE^d$.
Pick a generic $(\rho,\sigma) \in E$
(Definition~\ref{def:E}).

From Theorem~\ref{thm:pog}, $P(\rho,\sigma)=:(\tilde{\rho},\tilde{\sigma})$ 
is a pair of equivalent frameworks $C_{\SSS^d}(\Gamma)$
which are not congruent from Lemma~\ref{lem:gpnoncong}.

Let $\rho_\CC$ and $\sigma_\CC$ be the real valued complex frameworks
corresponding to $\rho$ and $\sigma$.
From Corollary~\ref{cor:CE}, 
$(\rho_\CC,\sigma_\CC)$ 
is generic in $E_\CC$.
Meanwhile,
$P_\CC(\rho_\CC,\sigma_\CC)= (\tilde{\rho}_\CC,\tilde{\sigma}_\CC)$,
where $\tilde{\rho}_\CC$ is the s-valued, complex representation of 
$\tilde{\rho}$,
and 
$\tilde{\sigma}_\CC$ is the s-valued, complex representation of 
$\tilde{\sigma}$.
From Corollary~\ref{cor:auto}, 
$\tilde{\rho}_\CC$ is generic in $C_{\CC^d}(\Gamma)$.
Therefore $\tilde{\rho}$ must be  generic
in $C_{\SSS^d}(\Gamma)$, and we can conclude 
that $\Gamma$ is not GGR in $\SSS^d$.

\end{proof}

\section{Proof of Theorem~\ref{thm:sudoGP}}

We will prove the theorem by first showing that 
the existence of a large enough GGR subgraph $\Gamma_0$ is sufficient
to rule out any ``cross-talk'' between different real signatures.
In particular, if we have an s-valued framework of $\Gamma_0$, 
then $\Gamma_0$  cannot have a
congruent  framework that is s'-valued where $s\neq s'$.
Thus, if we have an s-valued framework of $\Gamma$, 
then $\Gamma$  cannot have an
equivalent framework that is s'-valued where $s\neq s'$.
With  such cross talk ruled out, we will be able to apply an algebraic
degree argument to show that $\Gamma$ is GGF in $\SSS^d$.

In this section we will model congruence classes of frameworks 
in 
$C_{\CC^d}(\Verts)$ 
using complex symmetric matrices of rank $d$ or less. 
First we spell out some basic facts about these matrices,
and their relationship to configurations, as well
as the notions of congruence and equivalence.

\begin{definition}
\label{def:gram}
Let $\GG$ be the set of symmetric $v-1$ by $v-1$ complex matrices
of rank $d$ or less. This is a determinantal variety 
which is 
irreducible.
Assuming that $v \geq d+1$,
$\GG$ is  of complex dimension $v*d-\binom{d+1}{2}$,
and any generic $\MB \in \GG$ will have rank $d$.

For any configuration
$p \in C_{\CC^d}(\Verts)$ 
(or framework  $\rho \in C_{\CC^d}(\Gamma)$ )
we associate its \emph{g-matrix} $\gB(p) \in \GG$ 
as follows.
We first translate $p$ so its first vertex is at the origin.
For any two remaining vertices $t,u$, we define the 
corresponding matrix entry as
\begin{eqnarray}
\gB(p)_{t,u} := \sum_{i=1}^d p(t)_i \, p(u)_i
\end{eqnarray}
(This is like a Gram matrix, but there is no
conjugation involved). 
Overloading this notation, if $\rho$ is a framework with
configuration $p$, we define $\gB(\rho) := \gB(p)$.
\end{definition}

\begin{definition}
For any pair $\{t,u\}$,
of distinct vertices in $p$, there is a linear map $\pi_{t,u}$
that computes the squared lengths between that pair
using the entries in $\gB(p)$. 
In the case where 
 $t$ is the first vertex (that was mapped to the origin), we have
\begin{eqnarray}
\label{eq:orig}
\pi_{t,u}(\gB(p)) = \gB(p)_{u,u}
\end{eqnarray}
Otherwise, and in general,
\begin{eqnarray}
\label{eq:rest}
\pi_{t,u}(\gB(p)) = \gB(p)_{t,t} + \gB(p)_{u,u} -2 \gB(p)_{t,u}
\end{eqnarray}
Applying this to all pairs of distinct vertices 
induces a linear map $\pi_{K}$ 
from the set $\GG$ to the set of 
symmetric $v$ by $v$ complex matrices with zeros on the diagonal.
\end{definition}

\begin{lemma}
\label{lem:inj}
The map $\pi_K$ is injective.
\end{lemma}
\begin{proof}
We just need to show that the kernel of $\pi_K$ is $0$.
Let $\MB$ be a matrix in the kernel of $\pi_K$.
Starting with the first vertex at the origin, we find from
Equation (\ref{eq:orig}) that all of the diagonal entries, 
$\MB_{u,u}$ must vanish. Then, from Equation (\ref{eq:rest}), all the
off diagonal entries of $\MB$ must vanish as well.
\end{proof}

\begin{lemma}
\label{lem:gclass}
$p$ is congruent to $q$  iff
$\pi_K(\gB(p))=\pi_K(\gB(q))$  and  iff
$\gB(p)=\gB(q)$.
\end{lemma}
\begin{proof}
The first relation follows from the definition of congruence. 
The second follows from Lemma~\ref{lem:inj}.
\end{proof}

\begin{corollary}
\label{cor:inj}
The map $\gB$ acting 
on the quotient $C_{\CC^d}(\Verts)/{\rm congruence}$
is injective.
\end{corollary}

\begin{lemma}
\label{lem:gclose}
$\GG$ 
is 
the Zariski closure of 
$\gB(C_{\CC^d}(\Verts))$.
Moreover, 
if $p$ is generic in $C_{\CC^d}(\Verts)$, then 
$\gB(p)$ is generic in $\GG$.
\end{lemma}
\begin{proof}
Using Corollary~\ref{cor:inj},
a dimension count verifies that the image
$\gB(C_{\CC^d}(\Verts))$ 
must hit an open neighborhood of $\GG$ (ie. a subset of full dimension).
The results follow as $\GG$ is irreducible.
\end{proof}

Equivalence of frameworks can be defined through their
g-matrices as well:
\begin{definition}
Let $\pi_\Edges$ be 
the linear mapping from $\GG$ to $\CC^e$
defined by applying $\pi_{t,u}$ to each of the 
edges in $\Edges(\Gamma)$. 

$\rho$ is equivalent to $\sigma$, iff
$\pi_\Edges(\gB(\rho))=\pi_\Edges(\gB(\sigma))$.

If $\rho$ is generic in $C_{\CC^d}(\Gamma)$, then
(assuming $v \geq d+1$) 
$\pi_\Edges(\gB(\rho))$ is generic in $\pi_\Edges(\GG)$.

\end{definition}

The following Lemma will be useful when examining the cardinality
of a fiber of $\pi_\Edges$.
\begin{lemma}
\label{lem:conj}
Let $\MB$ be any matrix in $\GG$.
If $\pi_\Edges(\MB)$ is real valued, 
there must be an even number of non real matrices in 
$\pi_\Edges^{-1}(\pi_\Edges(\MB))$.
\end{lemma}
\begin{proof}
$\pi_\Edges$ is defined over $\RR$ and thus 
if $\MB_0$ is in $\pi_\Edges^{-1}(\pi_\Edges(\MB))$,  
so must its complex conjugate $\overline{\MB}_0$. 
If such an $\MB_0$ is not real,
then it is not equal to its conjugate.
\end{proof}

The following lemma is useful above in the proof of 
Lemma~\ref{lem:dimE}.

\begin{lemma}
\label{lem:equivgen}
Let $\Gamma$ be  generically locally rigid (in $\CC^d$).
Let $\rho$ be generic in 
$C_{\CC^d}(\Gamma)$. 
Let $\sigma$ be
equivalent to $\rho$. Then $\sigma$ is infinitesimally rigid. 
\end{lemma}
\begin{proof}
If  $\Gamma$ has less than  $d+2$ vertices and is generically locally rigid,
it must be a simplex, and we are done.

From 
Corollary~\ref{cor:inj}
and
Lemma~\ref{lem:gclose}, the set of congruence classes of configurations
has  dimension 
 $\dim(\GG)$, which is 
$v*d-\binom{d+1}{2}$.
Due to local rigidity, its measurement set, $\pi_\Edges(\GG)$,
has the same dimension.

Similarly, 
the set of frameworks with a degenerate affine span
must map to g-matrices with rank no greater than $d-1$,
and thus their measurement set must 
have dimension at most $v*(d-1) -\binom{d}{2}$.
Thus such degenerate measurements are non generic in 
$\pi_\Edges(\GG)$.

Meanwhile, the set of infinitesimally flexible frameworks
with non-degenerate span, is non generic in $C_{\CC^d}(\Verts)$, 
and so 
has dimension no larger than $v*d-1$.
Its measurement set has dimension no larger than
$v*d-1 -\binom{d+1}{2}$.
Thus the infinitesimally flexible measurements are non generic in the 
measurement set.

Thus a generic $\rho$ cannot map under the edge squared-length map
to any  measurement arising from an infinitesimally flexible
framework.

\end{proof}

-----

A real valued matrix in $\GG$ corresponds with an s-valued configuration.
At the heart of this correspondence is  Sylvester's law of inertia.
\begin{law}
\label{law:slaw}
Suppose $\MB$ is a real valued symmetric matrix
of size $v-1$ and 
rank $d$.
Suppose that $\MB = \BB^t \DB \BB$,
where $\BB$ is a real non-singular matrix, and 
where $\DB$ is a real diagonal matrix with
$s$ negative diagonal entries, $d-s$ positive diagonal entries,
and $v-1-d$ zero diagonal entries. Let us call the triple
$(s,d-s,v-1-d)$ 
the \emph{signature} of $\DB$.

Then $\MB$ cannot be written as
$\MB = \BB'^t \DB' \BB'$,  where $\BB'$ is real non-singular
and $\DB'$ is real diagonal with  a different signature.
Thus we can call $(s,d-s,v-1-d)$ 
the signature of $\MB$.

Since every real symmetric matrix has an orthogonal 
eigen-decomposition, it must have a signature.
\end{law}

\begin{lemma}
\label{lem:sig1}
Suppose some $\MB  \in \GG$ has all real entries
and has signature
$(s,d'-s,v-1-d')$ for some $s$ and $d'$ (with $d' \leq d$).
There exists an
s-valued configuration
$p$ with an affine span of dimension $d'$ and with
$\gB(p) = \MB$. 
\end{lemma}
\begin{proof}
By assumption $\MB=\BB^t \DB \BB$
where $\DB$ has signature
$(s,d'-s,v-1-d')$. Wlog, let us assume that 
the entries in $\DB$ appear in an order that matches the
signature.
Let us drop the last $v-1-d'$ rows of $\BB$.
Let us 
divide the $j$th row of $\BB$ by 
$\sqrt{|\DB_{j,j}|}$ to obtain an $d'$ by $v-1$ matrix $\PB'$.
Then we can write 
$\MB= \PB'^t \SB \PB'$, 
where $\SB$ is an $d'$ by $d'$ diagonal ``signature'' 
matrix with its first $s$ diagonal
entries of $-1$ and remaining $d'-s$ diagonal entries of $1$. 
Since $\BB$ is non-singular, $\PB'$ has rank $d'$.

Multiplying the first $s$ rows of $\PB'$
by $\sqrt{1}$, we can write $\MB= \PB^t  \PB$.
The columns of $\PB$ 
(along with
the origin) then give us the complex coordinates of an
s-valued  configuration 
$p \in C_{\CC^d}(\Verts)$ with $\gB(p)=\MB$.

\end{proof}

\begin{remark}
When $d'<d$, this does not rule out the possibility of other
frameworks with a different dimensional affine span, and different
real metric signature. When $d'=d$, Corollary~\ref{cor:sig3} 
(below) will
in fact rule out any other signatures and span dimensions.
\end{remark}

\begin{lemma}
\label{lem:sig0}
Let $p \in C_{\CC^d}(\Verts)$
be an s-valued configuration, 
then $\gB(p)$ is real.
If $p$ has an affine span of dimension $d' \leq d$, then 
$\gB(p)$ has rank no more than $d'$.
Moreover, if 
$p$ has an affine span of 
dimension $d$, then $\gB(p)$ has signature 
$(s,d-s,v-1-d)$.
\end{lemma}
\begin{proof}
Since $p$ is s-valued, 
$\gB(p)$ can be written in coordinates as $\PB'^t \SB \PB'$, where 
$\PB'$ is a $d$ by $v-1$ real matrix. 
And $\SB$ is a diagonal
matrix with $s$ entries of $-1$ and $d-s$ entries of $1$. 
The rank of $\gB(p)$ cannot exceed the rank of $\PB'$ which is $d'$.

If the affine span of $p$ has dimension $d$, then $\PB'$ has 
rank $d$.
Since the rows of $\PB'$ are linearly independent, we can use
those rows as the first $d$ rows of a non singular 
$v-1$ by $v-1$ matrix $\BB$.
We can use $\SB$ as the upper left block of a diagonal matrix
$\DB$ with the rest of the entries zeroed out.
Then we can write $\MB = \BB^t\DB \BB$ giving us the stated signature.
\end{proof}

\begin{corollary}
\label{cor:sig3}
Let $p \in C_{\CC^d}(\Verts)$
be an s-valued configuration with an affine span of 
dimension $d$. 
Let
$q \in C_{\CC^d}(\Verts)$ 
be an s'-valued configuration that is congruent to $p$. 
Then $q$ has an affine span of dimension $d$ and 
$s=s'$
\end{corollary}
\begin{proof}
From Lemma~\ref{lem:sig0}, $\gB(p)$ has signature $(s,d-s,v-1-d)$.
By the congruence assumption and Corollary~\ref{cor:inj},
, $\gB(p)=\gB(q)$.
As $\gB(q)$ has rank $d$, 
$q$ must have an affine span no less than $d$, and thus equal to $d$.
From Lemma~\ref{lem:sig0}, $\gB(q)$ must have 
signature $(s',d-s',v-1-d)$.
Thus $s=s'$.
\end{proof}

Now we can establish that when there is a GGR subgraph, 
the signature of all real matrices in a fiber of $\pi_\Edges$ is fixed.

\begin{lemma}
\label{lem:sameSig}
Let $\Gamma$ be  a graph and $\Gamma_0$ a GGR subgraph with 
$v_0$ vertices 
where $v_0 \geq d+1$.
Let $\rho$ be an s-valued framework 
in $C_{\CC^d}(\Gamma)$ 
for some $s$, with configuration $p$.
Suppose also that the affine span of the vertices of $\Gamma_0$ 
in $p$
is all of $\CC^d$.
Then all of the real matrices in the fiber
$\pi_\Edges^{-1} (\pi_\Edges(\gB(\rho)))$
must have signature
$(s,d-s,v-1-d)$.
\end{lemma}
\begin{proof}
Wlog, let $\Gamma_0$ include the first vertex,
and let its vertex set be $\Verts_0$.
We denote by $p_0$ the configuration $p$ restricted to 
$\Verts_0$. 
$p_0$, as a restriction of $p$, is s-valued.

Let $\MB$ be any real matrix in the fiber, and let it have 
signature $(s', d'-s', v-1-d')$ for some $s'$ and $d'$.
From Lemma~\ref{lem:sig1},
there must be some $q$, 
an s'-valued
configuration,  with $\gB(q)=\MB$.
When restricted to $\Verts_0$, the configuration $q_0$ must also
be s'-valued. Since $\Gamma_0$ is complex GGR, $p_0$ must be
congruent to $q_0$. Then from Corollary~\ref{cor:sig3}
$q_0$ must be s-valued and have affine span of dimension
$d$. Thus $s=s'$. Since $q$, as a super-set of
$q_0$, must have affine span of dimension
$d$, then from Lemma~\ref{lem:sig0}, $M$ must have
signature 
$(s,d-s,v-1-d)$.
\end{proof}

---

\begin{definition}
  Let $V\!$ and $W\!$ be irreducible complex algebraic sets 
of the same dimension
and $f : V \to W$
  be a surjective (or just dominant) algebraic map, all defined over~$\kk$.
Then the number of points in the fiber $f^{-1}(w)$ for any
generic $w \in W$ is a constant. This constant is called the
\emph{algebraic degree} of $f$.
\end{definition}

With this, we can complete the proof of Theorem~\ref{thm:sudoGP} by
applying a degree argument:
\begin{proof}
We will assume $\Gamma$ is generically locally rigid, 
otherwise we are already done.

Let $\rho$ be generic in 
$C_{\EE^d}(\Gamma)$.
From Lemma~\ref{lem:sig0}
$\gB(\rho)$ is real with signature $(0,d,v-1-d)$ (ie. it is PSD).
Because of the existence of a  GGR subgraph, 
from Lemma~\ref{lem:sameSig}, all of the real matrices in the fiber
$\pi_\Edges^{-1}(\pi_\Edges(\gB(\rho)))$ must have the same signature.
From Lemma~\ref{lem:sig0} and
Corollary~\ref{cor:inj}, these matrices are in 
one to one correspondence with the congruence classes
$[\rho_i]$ of 
equivalent frameworks in 
$C_{\EE^d}(\Gamma)$.
Since $\Gamma$ is not GGR, 
from Theorem~\ref{thm:ght}, there must be an 
even number of such classes and thus an even number
of real matrices in the fiber.

From  Lemma~\ref{lem:conj}, there are an even number of
non real matrices in the fiber and 
we see that the total cardinality
of $\pi_\Edges^{-1}(\pi_\Edges(\gB(\rho)))$ is even.
Since $\pi_\Edges(\gB(\rho))$ is generic in the image $\pi_\Edges(\GG)$,
this means that 
the \emph{algebraic degree} of $\pi_\Edges$ must be even.

Now suppose 
$\sigma$ is generic in $C_{\SSS^d}(\Gamma)$, 
which we model as a generic s-valued framework in $C_{\CC^d}(\Gamma)$.
$\gB(\sigma)$ is real valued and has signature $(s,d-s,v-1-d)$.
From Lemma~\ref{lem:sameSig}
all of the real matrices in the fiber 
$\pi_\Edges^{-1}(\pi_\Edges(\gB(\sigma)))$
must have the same  signature $(s,d-s,v-1-d)$.

Since $\gB(\sigma)$ is real, then so is $\pi_\Edges(\gB(\sigma))$
so from Lemma~\ref{lem:conj} there must be
an even number of non real matrices in the fiber 
$\pi_\Edges^{-1}(\pi_\Edges(\gB(\sigma)))$, and thus an even
number of real matrices in the fiber, all with signature
$(s,d-s,v-1-d)$.

From Lemma~\ref{lem:sig0} and
Corollary~\ref{cor:inj}, 
these are in 
one to one correspondence with the congruence classes
$[\sigma_i]$ of 
equivalent s-valued frameworks in 
$C_{\CC^d}(\Gamma)$.
Thus $\Gamma$ is generically globally flexible in $\SSS^d$.
\end{proof}

\begin{remark}
The reasoning  in  the above proof does not
hold when $\Gamma$ does not have the required GGR subgraph.
In particular, the non-GGR
graph $G_3$ of Jackson and Owen~\cite{owjack}
generically has an \emph{odd} number (namely $45$) of equivalent
complex realizations in $\CC^2$.
\end{remark}
\section{Extension to Hyperbolic Space}

\begin{figure}
  \centering
  \includegraphics[width=3in]{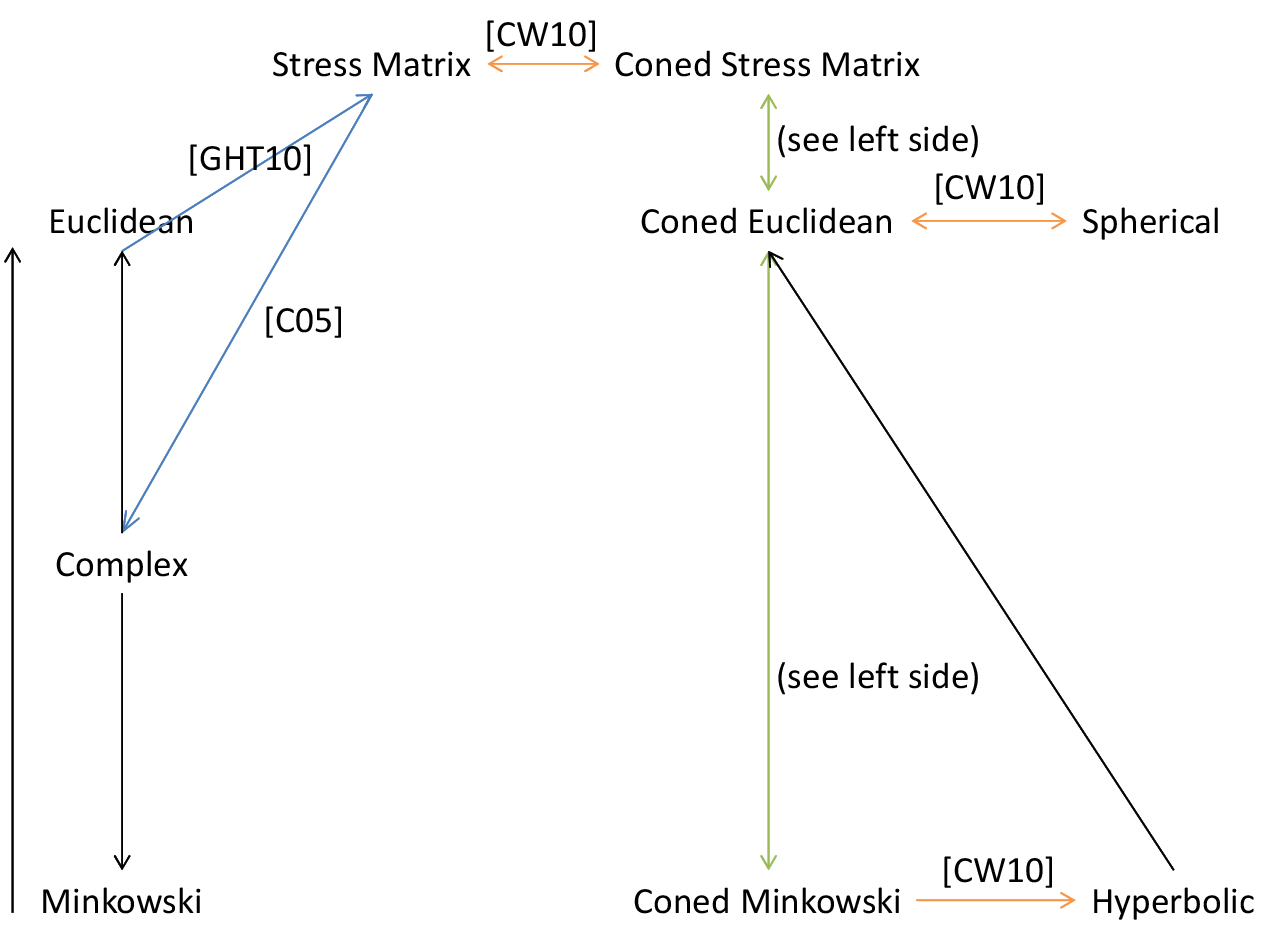}
  \caption{Implications between generic global rigidity in various spaces.
Black lines show implications proven in this paper.}
  \label{fig:arrows}
\end{figure}

Combining ideas from the previous section with results from
Connelly and Whiteley~\cite{connelly2010global},
we can transfer the property of generic global rigidity to
hyperbolic space $\HH^d$ as well.
\begin{corollary}
\label{cor:hyper}
A graph $\Gamma$ 
is generically globally rigid in $\EE^d$
iff it
is generically globally rigid in $\HH^d$.
\end{corollary}

This can be done using the coning operation explored in~\cite{connelly2010global}, and the proof is developed below.
\begin{definition}
Given a graph $\Gamma$ and a new vertex $u$, the \emph{coned graph}
$\Gamma * \{c\}$ is the graph obtained starting with $\Gamma$, adding
the vertex $c$ and adding an edge connecting $c$ to each vertex in 
$\Gamma$.
\end{definition}

\begin{theorem}[Connelly and Whiteley~\cite{connelly2010global}]
\label{thm:cone}
A graph $\Gamma$ 
is generically globally rigid in $\EE^d$
iff $\Gamma * \{c\}$
is generically globally rigid in $\EE^{d+1}$.
\end{theorem}
(This theorem is proven using an argument about equilibrium 
stress matrices. See Figure~\ref{fig:arrows}).

By modeling spherical d-space within a Euclidean d+1 space, 
Connelly and Whiteley then
show the equivalence between Euclidean GGR of $\Gamma * \{c\}$
and spherical GGR of $\Gamma$.

In a  similar manner, one can model hyperbolic space $\HH^d$ within 
the d+1 dimensional pseudo Euclidean space that has one negative coordinate
in its signature.
We denote this \emph{Minkowski space} as $\MM^{d+1}$.
In particular,  we model $\HH^d$ as the subset of 
vectors  $\vec{v} \in \MM^{d+1}$
such that $|\vec{v}|^2=-1$ under the Minkowski metric, 
and such that  $\vec{v}_1 > 0$, where $\vec{v}_1$ 
is the first coordinate of $\vec{v}$.
For two vectors  $\vec{v}$ and $\vec{w}$ on this 
``hyperbolic locus'', their  distance in $\HH^d$
corresponds to the arcosh of their Minkowski inner product.

\subsection{Proof of Corollary =>}

We begin with a hyperbolic lemma that mirrors a spherical 
lemma in~\cite{connelly2010global}.

\begin{lemma}
\label{lem:Htfer1}
Let $\rho$ and $\sigma$ be two equivalent and non congruent frameworks
of $\Gamma$
in $\HH^d$, then there is a corresponding pair $(\rho_\MM'',\sigma_\MM'')$
of equivalent and
non congruent 
frameworks of $\Gamma * \{c\}$ in $\MM^{d+1}$.
Moreover, if $\rho$ (or $\sigma$) is generic in $\HH^d$, then we can find a 
corresponding $\rho_\MM''$ (or $\sigma_\MM''$) that is  generic in 
$\MM^{d+1}$.
\end{lemma}
\begin{proof}
Given $\rho$ and $\sigma$, we model these as
$\rho_\MM$ and $\sigma_\MM$, two
frameworks of $\Gamma * \{c\}$ in $\MM^{d+1}$, with the cone vertex
$c$ at the origin and the rest of the vertices on the hyperbolic locus.
For each vertex $t\in \Verts(\Gamma)$,
we pick a generic positive scale $\alpha_t$ and multiply all of the
$d+1$
coordinates of $\rho_\MM(t)$ and $\sigma_\MM(t)$ by this $\alpha_t$.
Let us call the resulting pair,
$\rho_\MM'(t)$ and $\sigma_\MM'(t)$.
As in~\cite{connelly2010global}, $\rho_\MM'(t)$ and $\sigma_\MM'(t)$
are equivalent and non
congruent  in $\MM^{d+1}$. By translating these frameworks
by some generic offset, we obtain the desired pair
$\rho_\MM''$ and $\sigma_\MM''$.
\end{proof}

\begin{proof}[Proof of corollary =>]
Suppose a graph $\Gamma$ is not GGR in $\HH^d$ 
then from Lemma~\ref{lem:Htfer1},
$\Gamma * \{c\}$ 
is not GGR in $\MM^{d+1}$. 
Then From Theorem~\ref{thm:sudo},
$\Gamma * \{c\}$ is not GGR in $\EE^{d+1}$.
Then from Theorem~\ref{thm:cone},
$\Gamma$ is not GGR in $\EE^{d}$.
See Figure~\ref{fig:arrows}.
\end{proof}

\subsection{Proof of Corollary <=}

In order to prove the other direction
we restrict ourselves to Minkowski frameworks
that can be moved to the hyperbolic locus using positive scaling.

\begin{definition}
We say that a framework $\rho$ of $\Gamma * \{c\}$ in $\MM^{d+1}$
is  \emph{upper coned} if 
for all vertices 
$t \in \Verts(\Gamma)$, we have 
$|\rho(t)-\rho(c)|^2 < 0$ and
$(\rho(t)-\rho(c))_1 > 0$.
We say that $\rho$ 
is  \emph{lower coned} if 
for all vertices 
$t \in \Verts(\Gamma)$, we have 
$|\rho(t)-\rho(c)|^2 < 0$ and
$(\rho(t)-\rho(c))_1 < 0$.
\end{definition}

The following lemma is 
the needed partial converse of Lemma~\ref{lem:Htfer1}.

\begin{lemma}
\label{lem:Htfer2}
Let $\rho$ and $\sigma$ be two equivalent and non congruent frameworks
of $\Gamma * \{c\}$ in $\MM^{d+1}$.
And let us also assume that $\rho$ and $\sigma$ are upper coned.
Then there is a corresponding pair $(\rho_\HH,\sigma_\HH)$
of equivalent and
non congruent frameworks of $\Gamma$ in $\HH^{d}$.
Moreover, if $\rho$ (or $\sigma$) is generic in $\MM^{d+1}$, then 
$\rho_\HH$ (or $\sigma_\HH$) is generic in 
$\HH^{d}$.
\end{lemma}
\begin{proof}
Given $\rho$ and $\sigma$, we first translate the frameworks,
moving 
the cone vertex, $c$, to the origin in $\MM^{d+1}$. 
Let us call the resulting pair
$\rho'$ and $\sigma'$.
For each vertex $t\in \Verts(\Gamma)$, we then divide all of the
$d+1$ coordinates of $\rho'(t)$ and $\sigma'(t)$ by the positive
quantity,
$-|\rho(t)-\rho(c)|^2$ (which is the same as $-|\sigma(t)-\sigma(c)|^2$).
Let us call the resulting pair, $\rho''$ and $\sigma''$.
Due to our upper coned
assumption, these vertices all lie on the hyperbolic locus
and correspond to a pair of frameworks $\rho_\HH$ and $\sigma_\HH$ 
of $\Gamma$ in $\HH^d$.
As in~\cite{connelly2010global},
the resulting frameworks, $\rho_\HH$ and $\sigma_\HH$,
of $\Gamma$ are equivalent, non congruent,
and generic
in $\HH^d$.
\end{proof}

In order to ultimately get upper coned Minkowski frameworks, we
also define the following special framework classes.

\begin{definition}
We say that a framework $\rho$ of $\Gamma * \{c\}$ in $\EE^{d+1}$
is  \emph{spiky} if 
for one vertex
$t_0 \in \Verts(\Gamma)$, we have 
$|\rho(t_0)-\rho(c)| > 2$ 
and
for all edges
$(t,u) \in \Edges(\Gamma)$, we have 
$|\rho(t)-\rho(u)| < \frac{1}{v}$.
\end{definition}

\begin{definition}
We say that a framework $\rho$ of $\Gamma * \{c\}$ in $\FF^{d+1}$
is  \emph{upper cylindrical} if 
for all vertices 
$t \in \Verts(\Gamma)$, we have 
$(\rho(t)-\rho(c))_1 > 1$ 
and 
$\sum_{i=2}^{d+1}(\rho(t)-\rho(c))_i^2 < 1$.
\end{definition}

\begin{lemma}
\label{lem:sToU}
Let $\Gamma$ be connected.
If a framework $\rho$ of
$\Gamma * \{c\}$ in $\EE^{d+1}$
is  spiky, then it is related by rotation to a framework 
which is upper cylindrical.
\end{lemma}
\begin{proof}
We can find a rotation that moves $\rho(t_0)-\rho(c)$ onto the first axis,
with a first coordinate greater than $2$.
Since $\Gamma$ is connected, it has diameter no larger than $v$.
From the triangle inequality, all of the coordinates of all of the
vertices must satisfy the upper cylindrical conditions.
\end{proof}

\begin{lemma}
\label{lem:eToM}
Let $\rho$ and $\sigma$ be two upper cylindrical frameworks 
of $\Gamma * \{c\}$ in $\EE^{d+1}$. 
Then the resulting frameworks
from the Pogorelov map to $\MM^{d+1}$,
$(\tilde{\rho},\tilde{\sigma}) := P(\rho,\sigma)$, are both upper cylindrical.
\end{lemma}
\begin{proof}
This follows from directly the ``coordinate swapping'' interpretation
of the Pogorelov map from Remark~\ref{rem:cswap}.
\end{proof}

\begin{lemma}
\label{lem:uToC}
If a framework $\rho$ of
$\Gamma * \{c\}$ in $\MM^{d+1}$
is  upper cylindrical, then it is upper coned.
\end{lemma}
\begin{proof}
By definition, the first coordinates of all vertices have the 
required sign. 
Moreover, for any $t \in \Verts(\Gamma)$, 
\begin{eqnarray}
|\rho(t)-\rho(c))|^2 
=
-(\rho(t)-\rho(c)))_1^2 + 
\sum_{i=2}^{d+1}(\rho(t)-\rho(c)))_i^2 
< 
0 
\end{eqnarray}
And thus it is upper coned.
\end{proof}

With these simple facts established, we can now apply the 
machinery from Section~\ref{sec:sudo2} to the problem at hand.

\begin{lemma}
\label{lem:Mtfer}
Let 
$\Gamma * \{c\}$ 
be generically locally rigid in $\EE^{d+1}$. 
Suppose 
$\Gamma * \{c\}$ 
is not GGR in $\EE^{d+1}$, 
then $\Gamma * \{c\}$  
has an pair of generic frameworks in $\MM^{d+1}$,
that are
equivalent, non congruent, and upper coned. 
\end{lemma}
\begin{proof}
The proof follows that of Section~\ref{sec:sudo2}. The only issue
is ensuring the upper coned-ness 
of the result.

When picking the component $E$ (see Definition~\ref{def:E}) 
we choose a component of $E^+$ such that 
$E$ contains some non-congruent pair, 
$\dim(\pi_1(E))= v*d$, 
and such that 
$\pi_1(E)$ contains a framework $\rho$
that is spiky.

Since the set of frameworks
that are spiky is of dimension $v*d$, 
and by assumption, 
$\Gamma * \{c\}$ 
is not GGR in $\EE^{d+1}$,  and thus GGF in $\EE^{d+1}$, 
the projection
$\pi_1(E^+\backslash C^+)$ must include a set of spiky 
frameworks with dimension $v*d$.
Thus, at least one component with the stated properties must
exist. We will chose one such component and  will call it $E$.

Pick an $e := (\rho,\sigma) \in E$ in the fiber above $\rho$.
Since $\rho$ is spiky, and spikiness
only depends on edge lengths, $\sigma$ must be spiky as well. 
Next, we perturb $e$ in $E$ to get  $e' =: (\rho',\sigma')$ that is
generic in $E$. 
Since spikiness is an  open 
property, for small enough perturbations, both
$\rho'$ and $\sigma'$ will still be spiky.

Since 
$\Gamma * \{c\}$ 
is generically locally rigid in $\EE^{d+1}$,
$\Gamma$ must be connected.
Thus from Lemma~\ref{lem:sToU}, we can choose 
an upper cylindrical $\sigma'^c$ 
that is 
related to $\sigma'$ by rotation and translation 
as well as 
an upper cylindrical  $\rho'^c$ 
related to $\rho'$ by rotation.
From Lemma~\ref{lem:dimE}, since $e'$ is generic in $E$
the point $e'^c := (\rho'^c,\sigma'^c)$ 
must be in $E$ as well. 

Next we perturb $e'^c$ within $E$ to get $e^{\prime c\prime}
 =: (\rho^{\prime c\prime},\sigma^{\prime c\prime})$ 
which is generic in $E$.
Since upper cylindricality is an   open 
property, for small enough perturbations, both
$\rho^{\prime c\prime}$  
and $\sigma^{\prime c\prime}$ will still be upper cylindrical.

Now when we apply the Pogorelov map,
$(\widetilde{\rho^{\prime c\prime}},\widetilde{\sigma^{\prime c\prime}}) :=
P(e^{\prime c\prime})$.
As in the proof of Theorem~\ref{thm:sudo},
$\widetilde{\rho^{\prime c\prime}}$ and $\widetilde{\sigma^{\prime c\prime}}$
are equivalent, non
congruent and generic frameworks in 
$\MM^{d+1}$.
From Lemma~\ref{lem:eToM}
both 
$\widetilde{\rho^{\prime c\prime}}$ and $\widetilde{\sigma^{\prime c\prime}}$ 
must be upper cylindrical,
and from Lemma~\ref{lem:uToC},
both 
$\widetilde{\rho^{\prime c\prime}}$ 
and $\widetilde{\sigma^{\prime c\prime}}$ must be upper coned,
\end{proof}

\begin{proof}[Proof of corollary <=]
Suppose a graph $\Gamma$ is not GGR in $\EE^d$ 
then from Theorem~\ref{thm:cone},
$\Gamma * \{c\}$ 
is not GGR in $\EE^{d+1}$. 
Then from Lemma~\ref{lem:Mtfer},
$\Gamma * \{c\}$  
has an pair of generic frameworks in $\MM^{d+1}$
that are
equivalent, non congruent, and upper coned. 
Then from Lemma~\ref{lem:Htfer2},
$\Gamma$ 
is not GGR in $\HH^{d}$. 
\end{proof}

\begin{remark}
In Section 7 of~\cite{connelly2010global}, there is a brief sketch
describing how to directly use a Pogorelov type map to equate
Euclidean GGR and hyperbolic GGR. That discussion does not go into the
details showing that their construction hits an open neighborhood
of frameworks (ie. a generic framework), 
which is the main technical contribution of our
Theorem~\ref{thm:sudo}.
\end{remark}

\subsection{Hyperbolic GGF}

Using coning, 
we can also  prove a hyperbolic version of 
Theorem~\ref{thm:sudoGP}, namely:

\begin{corollary}
If a graph $\Gamma$ 
is not GGR in 
$\HH^d$, and it has 
a GGR subgraph $\Gamma_0$ with $d+1$ 
or more vertices, then $\Gamma$ 
must be GGF in $\HH^d$.
\end{corollary}

\begin{proof}
Having established that generic global rigidity transfers between
Pseudo Euclidean spaces and through coning, we know that
$\Gamma * \{c\}$, is not GGR in $\MM^{d+1}$.
Likewise, 
it has a coned subgraph with at least $d+2$ vertices,
$\Gamma_0 * \{c\}$, that is  GGR in $\MM^{d+1}$.
Thus, from Theorem~\ref{thm:sudoGP}, 
$\Gamma * \{c\}$  
must be GGF in  
$\MM^{d+1}$.

Let $\rho$ be a framework of $\Gamma$ in $\HH^d$.
We model this as
$\rho_\MM$, a
framework of $\Gamma * \{c\}$ in $\MM^{d+1}$, with the cone vertex
$c$ at the origin and the rest of the vertices on the hyperbolic locus.
For each vertex $t\in \Verts(\Gamma)$,
we pick a generic positive scale $\alpha_t$ and multiply all of the
$d+1$
coordinates of $\rho_\MM(t)$ by this $\alpha_t$.
Let us call the resulting framework
$\rho_\MM'(t)$.
By translating this frameworks
by some generic offset, we obtain
$\rho_\MM''$, 
a generic framework of the coned graph in $\MM^{d+1}$.
Since the $\alpha_t$ are all positive, 
$\rho_\MM''$ 
must be upper coned.

Since $\Gamma * \{c\}$  
is GGF in  
$\MM^{d+1}$, 
$\rho_\MM''$  must have an equivalent and non-congruent framework,
$\sigma_\MM''$.
From Lemma~\ref{lem:orthoCone} (below), 
we can choose $\sigma_\MM''$
to be upper coned.
Then from Lemma~\ref{lem:Htfer2}, 
there must be 
a framework, $\sigma$, 
of $\Gamma$ in $\HH^d$, 
that is equivalent
and non congruent to $\rho$.
\end{proof}

\begin{lemma}
\label{lem:orthoCone}
Let $\Gamma$ be a connected graph.
Let $(\rho,\sigma)$ be a pair of equivalent frameworks of
 $\Gamma * \{c\}$  
in  
$\MM^{d+1}$. Let us also assume that $\rho$ is in general position.
If $\rho$ is upper coned, then either $\sigma$ 
is upper coned or it is lower coned.
\end{lemma}
\begin{proof}
Let $t$ and $u$ be two vertices of $\Verts(\Gamma)$ that are
connected by an edge in $\Gamma$. Along with the edges $\{t,c\}$ 
and $\{u,c\}$, this defines  a triangle $T$, which is a subgraph of 
$\Gamma * \{c\}$.
Since $\sigma$ is equivalent to $\rho$, 
these frameworks when restricted
to $T$, must be, by definition, congruent. 

Since $\rho$ is in general position, from Corollary~\ref{cor:cong3}
these two frameworks 
of $T$ must be strongly congruent.
Thus, there is an orthogonal
transform of $\MM^{d+1}$ mapping
$({\rho(t) - \rho(c)})$ to 
$({\sigma(t) - \sigma(c)})$
and mapping 
$({\rho(u) - \rho(c)})$ to 
$({\sigma(u) - \sigma(c)})$.
An orthogonal transform either maps the entire upper cone to the 
upper cone, or it maps the entire upper cone to the lower cone.
Since $\Gamma$ is connected, this makes $\sigma$ either
upper coned or lower coned. (Moreover, by negating all of the coordinates in 
$\sigma$ we can always obtain an upper coned equivalent framework).
\end{proof}

\section{Algebraic Geometry Background}
\label{app:alg}

We start with some preliminaries 
from real
and complex
algebraic geometry, somewhat specialized to our particular case.  For
a general reference, see, for instance, the
book by Bochnak, Coste, and Roy~\cite{bcr}. Much of this is adapted
from~\cite{GHT10}.
\begin{definition}
  \label{def:generic-gen}
  An affine, real (resp. complex)
\emph{algebraic set} or \emph{variety}~$V$ 
  defined over a field~$\kk$
  contained in~$\RR$ (resp. $\CC$)
is a subset of $\RR^n$ (resp $\CC^n$)
that is
  defined by a set of algebraic equations with coefficients in~$\kk$.

 An algebraic set is closed in the Euclidean topology.

  An algebraic set is \emph{irreducible} if it is not the union of two proper
  algebraic subsets
  defined over $\RR$  (resp $\CC$). 
 Any reducible algebraic set $V$ can be 
uniquely described as the union of 
a finite number of maximal irreducible subsets called
the \emph{components} of $V$.

  A real (resp. complex) algebraic set has a 
real (resp. complex) \emph{dimension}
  $\dim(V)$, which we will define as the largest $t$ for which there
  is an open subset of~$V$, in the Euclidean topology, that is
isomorphic to $\RR^t$
(resp. $\CC^t$).
Any algebraic subset of an irreducible algebraic set must be of 
strictly lower dimension.

  A point~$x$ of an irreducible algebraic set~$V$ is \emph{smooth} 
(in the differential geometric sense)
if it has a
  neighborhood that is
  smoothly isomorphic to $\RR^{\dim(V)}$ (resp. $\CC^{\dim(V)}$).
 (Note that in a real variety,
there may be points
  with neighborhoods isomorphic to $\RR^n$ for some $n < \dim(V)$; we
  will not consider these points to be smooth.)

\end{definition}

\begin{definition}
  Let $\kk$ be a 
subfield of $\RR$.
  A \emph{semi-algebraic set}~$S$ defined over $\kk$ 
  is a subset of $\RR^n$ defined by algebraic
  equalities and inequalities with coefficients in $\kk$; 
  alternatively, it is the image of a real
  algebraic set (defined only by equalities) under an algebraic map
  with coefficients in $\kk$.
  A semi-algebraic set has a well defined (maximal) dimension~$t$.

  The real  \emph{Zariski closure} of $S$ is the smallest
real   algebraic set defined over $\RR$
containing it.  (Loosely speaking, we can get an 
algebraic set by keeping all algebraic
  equalities and dropping  the inequalities. We may 
need to enlarge the  field to cut out  the smallest algebraic set
containing~$S$
but a finite extension will always suffice.)

  We call $S$ \emph{irreducible} if its real Zariski closure is
  irreducible. 
  An irreducible semi-algebraic set $S$ has the same real dimension as its
  real Zariski closure.  

A point on $S$ is smooth
  if it has a neighborhood in $S$ smoothly isomorphic to
  $\RR^{\dim(S)}$. 

\end{definition}

\begin{lemma}
\label{lem:irrIm}
  The image of an irreducible real algebraic or semi-algebraic
set under a polynomial map
is  an irreducible semi-algebraic set.
  The image of an irreducible complex algebraic set under a polynomial map
is  an irreducible constructible set.
\end{lemma}

We next define genericity in larger generality and give some basic
properties.

\begin{definition}
  A point in
  a  (semi-)algebraic set~$V$ defined over $\kk$, 
a countable subfield  of $\RR$,
  is
  \emph{generic} if its
  coordinates do not satisfy
  any algebraic equation with coefficients in~$\kk$
  besides those that are satisfied by every
  point on~$V$.

  Almost every point in an irreducible (semi) algebraic set $V$ is generic.
\end{definition}

\begin{remark}
\label{rem:ext}
Note that the defining field might change when we take the real Zariski
closure $V$ of a semi-algebraic set $S$.
For example, in $\RR^1$, the single 
point $\sqrt{2}$ can be described using
equalities and inequalities with coefficients in $\QQ$, and thus  it
is semi-algebraic and defined
over $\QQ$. But as a real variety, the defining equation
for this single-point variety
requires coordinates in $\QQ(\sqrt{2})$. Indeed, the smallest variety that 
contains
the point 
$\sqrt{2}$ and that is defined over $\QQ$ 
must also include the point $-\sqrt{2}$.
However, this complication
does not matter for the purposes of genericity.

Specifically, if $\kk$ is a finite algebraic extension of $\QQ$
and $x$ is a generic point in an irreducible semi-algebraic set $S$ defined
over~$\kk$, 
then $x$
 is also generic in $V$,
the real Zariski closure of $S$, defined  over an appropriate field.
This follows from a three step argument. 
First, a dimensionality argument shows that $V$
must be a component of $V^+_\kk$, the smallest
real algebraic variety that is defined over $\kk$ and contains $S$.
Second, it is a standard algebraic fact that 
if a real (resp. complex) variety $W^+$
is defined over $\kk$, a subfield of $\RR$ 
(resp. $\CC$), then 
any of its components
is defined over 
some field~$\kk'$, a subfield of $\RR$ (resp. $\CC$),
which is a finite extension of~$\kk$.
Finally, from Lemma~\ref{lem:ext} (below), 
any non generic 
point $x \in V$ (ie. satisfying 
some algebraic equation with coefficients in~$\kk'$)
must also 
satisfy
some algebraic equation with coefficients in~$\kk$
(or even $\QQ$)
that is non-zero over~$V$.
\end{remark}

\begin{lemma}
\label{lem:ext}
Let $\kk'$ be some algebraic extension of $\QQ$.
Let $V$ be an irreducible algebraic set defined over $\kk'$.
Suppose a point $x\in V$
satisfies
an algebraic equation $\phi$ with coefficients in~$\kk'$
that is non-zero over $V$,
then $x$ must also 
satisfy
some algebraic equation $\psi$ with coefficients in~$\QQ$
that is non-zero over~$V$.
\end{lemma}
\begin{proof}
Let $H$ be the Galois group of the (normal closure of) $\kk'$ over
$\QQ$. For $h_i \in H$, denote $h_i(\phi)$ to be the polynomial
where $h_i$ is applied to each coefficient in $\phi$.
Let $A$ be the (possibly empty) ``annihilating  
set'',  such that $\forall h_i \in A$, $h_i(\phi)$ vanishes 
identically over $V$.

Let
\begin{eqnarray}
\phi^{\Sigma} := \phi + \sum_{h_i\in A} \lambda_i h_i(\phi)
\end{eqnarray}
(Where the $\lambda_i \in \QQ $ 
are simply an additional  set of blending weights ).

$\phi^{\Sigma}$ has the following properties:
\begin{itemize}
\item $\phi^{\Sigma}(x)=0$. 
\item 
(For almost every $\lambda$), 
for any $h \in H$, $h(\phi^{\Sigma})$ 
does not vanish identically 
over $V$. This follows since $h(\phi^{\Sigma})$ is made up of a sum
of $|A|+1$ polynomials, where no more than $|A|$ of them can
vanish identically over $V$. Under almost any 
blending weights $\lambda$, their sum will not cancel.
\end{itemize} 

Let
\begin{eqnarray}
\psi := \prod_{h_i\in H}  h_i(\phi^{\Sigma})
\end{eqnarray}

$\psi$ has the following properties:
\begin{itemize}
\item $\psi(x)=0$. 
\item $\psi$ does not vanish over $V$.
\item $h(\psi)=\psi$. Thus $\psi$ has coefficients in the fixed
field, $\QQ$.
\end{itemize}

\end{proof}

The following propositions are standard~\cite{GHT10}:

\begin{proposition}\label{prop:generic-smooth}
  Every generic point of a (semi-)algebraic set is smooth.
\end{proposition}

\begin{lemma}
\label{lem:onlyone}
Let $V^+$ be a (semi) algebraic set, not necessarily irreducible,
defined over $\kk$. Let $V$ be a component of $V^+$.
Let $x$ be generic in $V$.
Then $x$ 
does not lie on any other component of $V^+$.
Moreover, any point $x' \in V^+$ that is sufficiently close  to $x$
cannot lie on any other component of $V^+$. 
\end{lemma}
\begin{proof}
As per Remark~\ref{rem:ext} any component must be defined over an 
algebraic extension of $\kk$. The  defining equations of any other 
component would 
produce
an equation obstructing the genericity of $x$  in $V$.
Since a variety is a closed set in the Euclidean topology, 
no other component of $V^+$ can approach $x$.
\end{proof}

\begin{lemma}
  \label{lem:image-generic}
  Let $V\!$ and $W\!$ be (semi) algebraic sets 
with $V$ irreducible, 
and let $f : V \to W$
  be a surjective
(or just dominant) 
algebraic map (ie. where each of the coordinates of $f(x)$ is a some
polynomial expression in the coordinates of $x$), 
all defined over~$\kk$.  Then if $x
  \in V\!$ is generic, $f(x)$ is generic inside~$W\!$.
\end{lemma}

\begin{definition}
\label{def:cxify}
The \emph{complexification} $V^*$ of a real variety $V$ is the smallest
complex variety that contains $V$~\cite{Whitney57:ElemStructRealVarieties}. 
The complex dimension of $V^*$ is equal to the real dimension
of $V$. If $V$ is irreducible, then so is $V^*$. 
If $V$ is defined over $\kk$, so is $V^*$.
A generic point in $V$ is also generic in $V^*$.
\end{definition}

\section{Congruence}
\label{app:cong}

The following material is standard  and 
is included here for completeness.
This presentation is adapted from~\cite{inv,invThesis}.

In \emph{all} discussions in this section, we will assume that we have
first  translated
any configuration, say  
$p \in C_{\CC^d}(\Verts)$ 
so that its first vertex lies at the origin.
We  then treat the rest of the vertices as vectors in $\CC^d$,
and call them the \emph{vectors of $p$}.

\begin{definition}
We define the symmetric billinear form $\beta(\vec{v},\vec{w})$ 
over pairs of vectors,
$\{\vec{v},\vec{w}\}$
in $\CC^d$ as $\beta(\vec{v},\vec{w}) := \VB^t \WB$ where $\VB$ is the 
$d$ by $1$ (canonical) coordinate vector
of $\vec{v}$. (No conjugation is used here).
If $O$ is an orthogonal transformation on $\CC^d$, we have
$\beta(\vec{v},\vec{w})=\beta(O(\vec{v}),O(\vec{w}))$.

$\beta$ is non degenerate: there is no non-zero vector,
 $\vec{v}$, such that 
$\beta(\vec{v},\vec{w})=0$ for all $\vec{w}\in \CC^d$.

The squared length of a vector 
$\vec{v}$ is simply $\beta(\vec{v},\vec{v})$

With this notation, the $v-1$ by $v-1$ 
g-matrix 
has entries 
$\gB(p)_{t,u} = \beta(\bvec{p(t)},\bvec{p(u)})$.

For the case of 
the pseudo Euclidean space $\SSS^d$ 
we define $\beta(\vec{v},\vec{w}) := \VB^t \SB \WB$, where 
$\SB$ is the 
$d$ by $d$ diagonal
``signature matrix''
 having its first  $s$ diagonal entries $-1$, and the remaining
diagonal entries $1$.
\end{definition}

\begin{lemma}
\label{lem:nonsing}
Let $p_0$ 
be a  configuration of $d+1$ points in $\CC^d$,
with affine span of dimension $d$. 
Then 
$\gB(p_0)$ has rank $d$.  
The same is true in a pseudo Euclidean space $\SSS^d$.
\end{lemma}
\begin{proof}
The 
matrix 
$\gB(p_0)$ represents the form $\beta$, over all of $\CC^d$,
expressed in the 
basis defined by the vectors of $p_0$.
Since $\beta$ is a non-degenerate form, 
$\gB(p_0)$ must have
rank $d$.
\end{proof}

\begin{lemma}
\label{lem:cong1}
Let $p_0$ and $q_0$ 
be two congruent configurations of $a+1$ points in $\CC^d$,
both with 
affine span of dimension $a$. 
Then $p_0$ is strongly congruent to  $q_0$.
The same is true in a pseudo Euclidean space $\SSS^d$.
\end{lemma}
\begin{proof}
Since the  vectors of 
$p_0$ and $q_0$ are in general linear position, 
we can find an invertible 
linear
transform   $O_0$ such that, for all of the vectors of $p_0$
and $q_0$, indexed by a vertex $t$, 
we have
$\bvec{q(t)}  = O_0(\bvec{p(t)})$.
(The action of $O_0$ is uniquely defined between 
$\spanop(p_0)$ and $\spanop(q_0)$,
the a-dimensional linear spaces spanned by the vectors of $p_0$
and the vectors of $q_0$.)

The 
matrix 
$\gB(p_0)$ represents the form $\beta$,
restricted to $\spanop(p_0)$,
expressed in the 
basis defined by the vectors of $p_0$, 
while 
$\gB(q_0)$ represents $\beta$,
restricted to $\spanop(q_0)$,
expressed in the 
basis defined by the vectors of $q_0$, 
Since $\gB(p_0)=\gB(q_0)$, the map $O_0$ must act as an isometry between 
all of $\spanop(p_0)$ and $\spanop(q_0)$.

If $a=d$ we are done. 
Otherwise, from Witt's theorem (see~\cite{inv}),  the isometric
action of 
$O_0$ between $\spanop(p_0)$ and $\spanop(q_0)$ can be 
can be
extended to an isometry, $O$, acting on all of $\CC^d$.
Thus $p_0$ and $q_0$ must be strongly congruent.
\end{proof}

\begin{lemma}
\label{lem:cong2}
Let $p$ and $q$ 
be two congruent configurations of $v$ points in $\CC^d$,
both with affine span of dimension $a$. Suppose also that 
$\gB(p)=\gB(q)$ has rank $a$.  
Then  $p$ is strongly congruent to  $q$.
The same is true in a pseudo Euclidean space $\SSS^d$.
\end{lemma}
\begin{proof}
Since $\gB(p)$ has rank $a$, it must have some $a$ by $a$ non-singular
principal submatrix, associated with 
a 
subset of $a$ vertices.
The vertices in this subset must have a linear span of dimension $a$ in
both $p$ and $q$.
We denote by $p_0$ the configuration $p$ restricted to the
$a+1$ vertices comprised of 
this subset
together with the first vertex (at the origin).
And  likewise for $q_0$.
From Lemma~\ref{lem:cong1} there must be an isometry $O$ of
$\CC^d$, such that for any vertex $t$ in $p_0$, we have
$\bvec{q_0(t)}=O(\bvec{p_0(t)})$.

For  any vertex $u \in \Verts$, by our assumption on the dimension
of the affine span of $p$ and $q$, we have 
$\bvec{p(u)} \in \spanop(p_0)$ and 
$\bvec{q(u)} \in \spanop(q_0)$.
Since $\gB(p_0)=\gB(q_0)$ is invertible, 
the coordinates of 
$\bvec{p(u)}$ with respect to the basis
$p_0$, can be determined  from the appropriate
entries in 
$\gB(p)$. 
Likewise, the coordinates of 
$\bvec{q(u)}$ with respect to the basis 
$q_0$, can be determined from
$\gB(q)$. 
Since $\gB(p)=\gB(q)$ these coordinates must be the same.
Thus $\bvec{q(u)} = O(\bvec{q(u)})$,
and
$p$ and $q$ are  strongly congruent.

\end{proof}

\begin{corollary}
\label{cor:cong3}
Let $p$ and $q$ 
be two congruent configurations of $v \geq d+1$ points in $\CC^d$,
both with a d-dimensional affine span.
Then  $p$ is strongly congruent to  $q$.
If  $v < d+1$, and
$p$ and $q$ are in general position,
then  $p$ is strongly congruent to  $q$.

The same is true in a pseudo Euclidean space $\SSS^d$.
\end{corollary}
\begin{proof}

For the first statement, 
we can pick 
$d$ vertices, together with the first vertex at the origin,
to form a subset of size $d+1$,
that has  a linear span of dimension $d$ in
$p$. 
We denote by $p_0$ the configuration $p$ restricted to this subset. 
From Lemma~\ref{lem:nonsing}, 
the principal submatrix of $\gB(p)$ associated with this basis
must have rank $d$. 
The result then follows
from Lemma~\ref{lem:cong2}.

If $v \leq d+1$ and the points are in general position, then 
the result follows directly from Lemma~\ref{lem:cong1}.

\end{proof}

\bibliographystyle{acm}
\bibliography{alggeom,graphs,topo}
\end{document}